\documentclass[a4paper,12pt]{article}
\usepackage[utf8]{inputenc}
\usepackage[T1]{fontenc}
\usepackage[english]{babel}
\usepackage{amsmath, amsthm, amssymb}
\usepackage{mathrsfs}
\usepackage{enumerate}

      \theoremstyle{plain}
      \newtheorem{theorem}{Theorem}[section]
      \newtheorem*{theorem*}{Theorem}
      \newtheorem{lemma}[theorem]{Lemma}
      \newtheorem*{lemma*}{Lemma}
      \newtheorem{corollary}[theorem]{Corollary}
      \newtheorem*{corollary*}{Corollary}

      \theoremstyle{definition}
      \newtheorem{definition}[theorem]{Definition}
      \newtheorem*{definition*}{Definition}

      \theoremstyle{remark}
      \newtheorem{remark}[theorem]{Remark}
      \newtheorem*{remark*}{Remark}
      
      \newtheorem*{example*}{Example}

\DeclareMathOperator{\supp}{supp}

\newcommand{\wbar}[1]{\overline{#1}}
\newcommand{\what}[1]{\widehat{#1}}

\newcommand{\Ca}{\mathscr{C}}
\newcommand{\Cab}{\wbar{\mathscr{C}}}
\newcommand{\C}{\mathbb{C}}
\newcommand{\R}{\mathbb{R}}

\newcommand{\Tr}{\operatorname{Tr}}
\renewcommand{\Re}{\operatorname{Re}}
\renewcommand{\Im}{\operatorname{Im}}

\renewcommand{\O}{\Omega}
\renewcommand{\d}{\partial}
\newcommand{\db}{\wbar{\partial}}
\newcommand{\abs}[1]{\left\lvert #1 \right\rvert}
\newcommand{\norm}[1]{\left\lVert #1 \right\rVert}
\def\clap#1{\hbox to 0pt{\hss#1\hss}}

\def\mathrlap{\mathpalette\mathrlapinternal}

\def\mathrlapinternal#1#2{\rlap{$\mathsurround=0pt#1{#2}$}}

\title{Stability and uniqueness for the inverse problem of the Schr\"odinger equation in 2D with potentials in $W^{\varepsilon,p}$}
\author{Eemeli Bl{\aa}sten}
\date{}

\begin{document}
\maketitle
\section{Forewords}
This result will be published as part of my PhD thesis later. This manuscript contains the proof of the claim, but is not peer-reviewed. The proof can still be streamlined, especially by proving a better version of lemma \ref{boundaryInt}. This will require interpolation between Lorentz and BMO spaces in a domain.

\section{Abstract}
We prove uniqueness and stability for the inverse problem of the 2D Schr\"odinger equation in the case that the potentials give well posed direct problems and are in $W^{\varepsilon,p}(\Omega)$, $\varepsilon>0$, $p>2$. The idea of the proof is to use Bukhgeim's oscillating solutions $e^{in(z-z_0)^2}f$, $e^{in(\wbar{z}-\wbar{z_0})^2}g$. By Alessandrini's identity and stationary phase we get information about $q_1-q_2$ at $z_0$ from the Dirichlet-Neumann maps $\Lambda_{q_1}-\Lambda_{q_2}$.

Using interpolation, we see that the the worst of the remainder terms decays like $n^{1 - \varepsilon - \beta}$. Here $q_j \in W^{\varepsilon,p}$ and $\beta$ is the exponent in the norm estimate for the conjugated Cauchy operator in theorem \ref{BIGTHM}. We get $\beta$ arbitrarily close to $1$, so have uniqueness and stability for $\varepsilon > 0$.

The main inspiration for this proof has come from three different sources: \cite{bukhgeim}, \cite{alessandrini} and the lecture notes \cite{salo}. For technical details we have mainly used \cite{ONeil}, \cite{BL}, \cite{triebel1}.

\newpage
\section{Notation and general remarks}
\begin{itemize}
\item We denote the unit disc in $\C$ by $\O$.
\item Given $p\in\R$ we denote by $p*$ the number whose Sobolev conjugate $p$ is: $\frac{1}{p*} = \frac{1}{2} + \frac{1}{p}$.
\item All the norms are taken in $\O$ unless otherwise specified.
\item We may write for example $L^p(\O,z_0)$ to specify that the norm is taken with respect to $z_0$.
\item Some spaces we are going to use
\subitem $L^p$: the standard Lebesgue space of index $p \in [1,\infty]$.
\subitem $W^{k,p}$, $k$ integer: the space of $L^p$ functions whose distribution derivatives of order up to $k$ are also in $L^p$
\subitem $L^{(p,q)}$, $p>1$, $0<q\leq \infty$: the Lorentz space (with norm), as defined in \cite{ONeil}
\subitem $W^{s,p}$ for $s \in \R$: The Sobolev spaces as restrictions to $\O$ of the ones defined in \cite{BL}
\subitem $C^k(\wbar{\O})$, $k$ integer: the space of uniformly continuous functions on $\wbar{\O}$ whose derivatives of order up to $k$ are also uniformly continuous on $\wbar{\O}$
\item We don't always write the whole symbol for the space when taking the norm: 
\subitem $\norm{\cdot}_p$ denotes the $L^p$ norm
\subitem $\norm{\cdot}_{s,p}$ denotes the $W^{s,p}$ norm
\subitem $\norm{\cdot}_{(p,q)}$ denotes the $L^{(p,q)}$ norm
\item Interpolation spaces: In $X_\theta$ and $X'_\theta$ the variable of the continuous space is usually $z_0$.
\subitem $A_\theta = F_\theta(L^p,W^{1,p})$
\subitem $X_\theta = F_\theta\big(C^0(\wbar{\O},L^p), C^0(\wbar{\O},W^{1,p})\big)$
\subitem $X_\theta = F_\theta\big(C^0(\wbar{\O},L^\infty), C^0(\wbar{\O},W^{1,p})\big)$
\item By expressions like $qf$, $q\in A_\theta$, $f\in X'_\theta \cup X'_\theta$ we mean the element $\tilde{q}f$, where $\tilde{q}(z_0) = q$ for all $z_0$.
\end{itemize}

\section{Stationary phase method}

\begin{lemma}[Mean-value inequality]
\label{MVI}
Let $f : X \to \C$, $X \subset \C$ be convex, $f \in C^1(\wbar{X})$. Then for all $x,y \in X$
\begin{equation}
\abs{f(x) - f(y)} \leq
	\begin{cases}
		\norm{\sqrt{\abs{\d_1 f}^2 + \abs{\d_2 f}^2}}_{L^\infty(X)} \abs{x - y} \\
		\sqrt{2} \norm{\sqrt{\abs{\d f}^2 + \abs{\db f}^2}}_{L^\infty(X)} \abs{x - y}
	\end{cases}.
\end{equation}
\end{lemma}

\begin{proof}
By \cite[Thm. 7.20]{rudin} we have
\begin{equation}
\begin{split}
\lvert &f(x) - f(y)\rvert = \abs{\int_0^1 \frac{d}{dt} f \left(tx+(1-t)y\right) dt} \\
&= \abs{\int_0^1 \left( \Re \nabla f \cdot (x_1-y_1,x_2-y_2) + i \Im \nabla f \cdot (x_1-y_1,x_2-y_2) \right) dt} \\
&\leq \int_0^1 \sqrt{ \abs{\Re \nabla f}^2 + \abs{\Im \nabla f}^2 } \, \abs{x - y} dt \leq \norm{\abs{\nabla f}}_{L^\infty(X)} \abs{x - y}.
\end{split}
\end{equation}
Note that $\abs{\Re \nabla f}^2 + \abs{\Im \nabla f}^2 = \abs{\d_1 f}^2 + \abs{\d_2 f}^2 = 2\left( \abs{\d f}^2 + \abs{\db f}^2 \right)$, from which the claim follows.
\end{proof}

\begin{lemma}
\label{GaussianHolderNorm}
Let $\alpha \geq 0$ and $\xi \in \C$. Then we have
\begin{equation}
\abs{1 - e^{-i(\xi^2+\wbar{\xi}^2)}} \leq 2^{1+\alpha/2}\abs{\xi}^\alpha.
\end{equation}
\end{lemma}
\begin{proof}
A direct calculation, the two cases to consider are $\abs{\xi} < \frac{1}{\sqrt{2}}$ and $\abs{\xi} \geq \frac{1}{\sqrt{2}}$. We use lemma \ref{MVI} to get the first case.
\begin{equation}
\begin{split}
\sup_{\abs{\xi} \leq 2^{-1/2}} & \frac{\abs{1 - e^{-i(\xi^2 + \wbar{\xi}^2)}}}{\abs{\xi}^\alpha} \\
&\leq  \sup_{\abs{\xi} \leq 2^{-1/2}} \sqrt{2} \norm{\sqrt{\abs{-2ize^{-i(z^2+\wbar{z}^2)}}^2 + \abs{-2i\wbar{z}e^{-i(z^2+\wbar{z}^2)}}^2}}_{\mathrlap{L^\infty(2^{-1/2}\O)}} \,\,\, \abs{\xi}^{1-\alpha} \\
&\leq \sqrt{2}\cdot 2 \cdot \sqrt{2} \cdot 2^{-1/2} (2^{-1/2})^{1-\alpha} \leq 2^{1 + \alpha/2}
\end{split}
\end{equation}
The second case follows because $\xi^2 + \wbar{\xi}^2 \in \R$.
\begin{equation}
\sup_{\abs{\xi} \geq 2^{-1/2}} \frac{\abs{1 - e^{-i(\xi^2 + \wbar{\xi}^2)}}}{\abs{\xi}^\alpha} \leq \sup_{\abs{\xi} \geq 2^{-1/2}} \frac{2}{\abs{\xi}^\alpha} = 2^{1+\alpha/2}
\end{equation}
\end{proof}

Next we denote $R = (z-z_0)^2 + (\wbar{z} - \wbar{z_0})^2$, where $z_0$ is a point in $\C$.

\begin{lemma}[Stationary phase]
\label{stationaryPhase}
Let $Q\in W^{\alpha,2}(\C)$, $\alpha \geq 0$, $n > 0$. Then
\begin{equation}
\norm{Q - \frac{2n}{\pi} \int_\C e^{inR}Q(z)\,dm(z)}_{L^2(\C,z_0)} \leq C_\alpha n^{-\alpha/2} \norm{Q}_{W^{\alpha,2}(\C)},
\end{equation}
where $C_\alpha < \infty$.
\end{lemma}
\begin{proof}
A direct calculation using the Fourier transform and lemma \ref{GaussianHolderNorm}:
\begin{multline}
\norm{Q - \frac{2n}{\pi} \int_\C e^{inR}Q(z)\,dm(z)}_{L^2(\C,z_0)} = \norm{ \what{Q} - e^{-i\frac{\xi^2 + \wbar{\xi}^2}{16n}} \what{Q} }_{L^2(\C)} \\
\leq 4^{-\alpha} n^{-\alpha/2} \norm{\frac{\Big|1-e^{-i\big((\frac{\xi}{4\sqrt{n}})^2 + (\wbar{\frac{\xi}{4\sqrt{n}}})^2\big)}\Big|} {\big|\frac{\xi}{4\sqrt{n}}\big|^\alpha} \, \abs{\xi}^\alpha \what{Q}}_{L^2(\C)} \\
\leq 2^{1-3\alpha/2} n^{-\alpha/2} \norm{\abs{\xi}^\alpha \what{Q}}_{L^2(\C)} \leq C_\alpha n^{-\alpha/2} \norm{Q}_{W^{\alpha,2}(\C)}
\end{multline}
\end{proof}

\section{Bukhgeim type solutions}
We prove the existence of Bukhgeim's solutions and give some norm estimates for them. By $\Ca$ and $\wbar{\Ca}$ we denote the Cauchy-operators (convolution with $z^{-1}$ and $\wbar{z}^{-1}$, respectively). All the norms taken here are in $\O$.

We use interpolation theory to prove a norm estimate for the remainder terms in an intermediate space between $L^p$ and $W^{1,p}$. This estimate is of the form $\norm{r}_\theta \leq n^{-\beta} \norm{q}_\theta$, where $\beta$ does not depend on $\theta$. This $\beta$ will in fact give the speed at which the modulus of continuity in the stability estimate goes to zero when the potentials have one Sobolev derivative. The only place where we require smoothness is when integrating by parts. Thus if $\beta > 0$ we have integrated by parts too much, because we could get stability with a smaller value of $\beta$.

The main point is that if we have a stability estimate with a modulus of continuity, we may worsen that modulus to let the potentials be in a bigger space.

\bigskip
First we prove some estimates for the Cauchy-operators.

\begin{lemma}
\label{ONeilLemma1}
Let $2< p < \infty$, $\frac{1}{p*} = \frac{1}{p} + \frac{1}{2}$. Then there are $C,C_p < \infty$ suth that if $f \in L^{p*}$ then
\begin{equation}
\norm{\Ca f}_p \leq C_p \norm{f}_{p*}
\end{equation}
and for $f \in L^{(2,1)}(\O)$ we have
\begin{equation}
\norm{\Ca f}_{L^\infty(\O)} \leq C \norm{f}_{L^{(2,1)}(\O)}.
\end{equation}
\end{lemma}
\begin{proof}
By \cite[thm 2.6]{ONeil} we have for $f \in L^{(p_1,q_1)}$, $g\in L^{(p_2,q_2)}$,
\begin{equation}
\frac{1}{p_1}+\frac{1}{p_2} > 1, \quad \frac{1}{p_1}+\frac{1}{p_2} - 1 = \frac{1}{r}, \quad  s \geq 1 \quad \text{such that } \frac{1}{q_1} + \frac{1}{q_2} \geq \frac{1}{s}
\end{equation}
the norm estimate
\begin{equation}
\norm{ f \ast g}_{L^{(r,s)}} \leq 3r \norm{f}_{L^{(p_1,q_1)}} \norm{g}_{L^{(p_2,q_2)}}.
\end{equation}
Here $L^{(a,b)}$ denotes the Lorenz spaces (with norm). Moreover the same article states that
\begin{equation}
\norm{f \ast g}_\infty \leq \norm{f}_{L^{(p_1,q_2)}} \norm{g}_{L^{(p_2,q_2)}}
\end{equation}
if $\frac{1}{p_1} + \frac{1}{p_2} = 1$ and $\frac{1}{q_1} + \frac{1}{q_2} \geq 1$.

Note that the Cauchy operators are convolutions with $\frac{1}{\pi z}$, which is in $L^{(2,\infty)}$. Choose $p_1 = 2$, $q_1 = \infty$, $p_2 = q_2 = p*$ and $r = s = p$. This implies the the first claim, because for $1<a\leq \infty$ we have $L^{(a,a)} = L^a$. Then choose $p_1 = 2$, $q_1 = \infty$, $p_2 = 2$, $q_2 = 1$ to get the second claim.
\end{proof}

\begin{definition}
Let $B$ be a Banach space. Then the \emph{space of uniformly continuous $B$-valued functions} is
\begin{equation}
C^0(\wbar{\O}, B) = \{ f:\wbar{\O} \to B \mid \text{$f$ is pointwise continuous at each $z_0 \in \wbar{\O}$} \},
\end{equation}
equipped with the norm $\norm{f}_{C^0(B)} = \sup_{z_0} \norm{f(z_0)}_B$.
\end{definition}
\begin{lemma}[Well-definedness]
\label{wellDf}
If $B$ is a Banach space then $C^0(\wbar{\O},B)$ is a Banach space.
\end{lemma}
\begin{proof}
The proof is exactly the same as for $C^0(\wbar{\O}, \C)$ and can be found in almost any elementary book on functional analysis.
\end{proof}

Here we construct a test function which we will need for the most important theorem of this section (thm \ref{BIGTHM}).
\begin{lemma}
\label{testF}
Let $0<\delta<1$. Then there exists a function $h \in C^\infty(\wbar{\O}^2)$ which satisfies
\begin{enumerate}
\item $0\leq h \leq 1$
\item $h(z) = 0 \Leftrightarrow \abs{z-z_0} \leq \delta/2$
\item $m(\supp(1-h)) \leq \pi \delta^2$ for all $z_0 \in \wbar{\O}$
\item $\sup_{z_0} \norm{\frac{h}{\wbar{z}-\wbar{z_0}}}_{W^{1,l}(\O)} \leq c_l \delta^{2(\frac{1}{l}-1)}$ for all $1< l < 2$.
\end{enumerate}
\end{lemma}
\begin{proof}
Let $H\in C^\infty(\C)$ be such that $0 \leq H \leq 1$, $H(z) = 0 \Leftrightarrow \abs{z} \leq \frac{1}{2}$ and $H(z) = 1 \Leftrightarrow \abs{z} \geq 1$. Then define $h(z) = H(\frac{z-z_0}{\delta})_{|\wbar{\O}}$. Now clearly $h \in C^\infty(\wbar{\O}^2)$ and conditions 1, 2 and 3 are satisfied.

Let us calculate the $L^l(\O)$ norms of the function and its derivaties. Keep in mind that $0<\delta<1$ and $1< l$ so
\begin{multline}
\norm{\frac{h}{\wbar{z}-\wbar{z_0}}}_l^l = \int_\O \abs{\frac{H(\frac{z-z_0}{\delta})}{\wbar{z}-\wbar{z_0}}}^l dm(z) \leq \int_{2\O} \abs{\frac{H(z/\delta)}{z}}^l dm(z) \leq \int_{2\O \setminus \frac{\delta}{2}\O} \abs{z}^{-l} dm(z) \\
= 2\pi \int_{\delta/2}^2 r^{1-l} dr = 2\pi \frac{2^{2-l}-(\delta/2)^{2-l}}{2-l} \leq 2\pi \frac{2^{2-l}}{2-l} \leq C_l^l \delta^{2(1 - l)}.
\end{multline}
Then the derivatives. Note that $\partial \frac{h}{\wbar{z}-\wbar{z_0}} = \frac{\partial h}{\wbar{z}-\wbar{z_0}}$ so we only do the calculations for $\db$. They go similarly for $\d$ but with one term less. The first term
\begin{equation}
\begin{split}
\norm{\frac{\db h}{\wbar{z}-\wbar{z_0}}}_l^l &= \int_\O \abs{\frac{\db\big(H(\frac{z-z_0}{\delta})\big)}{\wbar{z}-\wbar{z_0}}}^l dm(z) \leq \int_{2\O} \abs{\frac{\db(H(z/\delta))}{z}}^l dm(z) \\
&\leq 2^l \delta^{-l} \norm{\nabla H}_\infty^l \int_{\delta\O \setminus \frac{\delta}{2}\O} \abs{z}^{-l} dm(z) \leq C_l^l \delta^{-l} \int_{\delta/2}^\delta r^{1-l}dr \\
&\leq C_l^l \delta^{-l} \frac{\delta^{2-l}-(\delta/2)^{2-l}}{2-l} \leq {C'_l}^l \delta^{2(1-l)}.
\end{split}
\end{equation}
And finally the last term. Note that $z_0 \notin \supp h$ so the singularity of $1/(\wbar{z}-\wbar{z_0})$ does not cause problems:
\begin{equation}
\begin{split}
\norm{h \db \frac{1}{\wbar{z}-\wbar{z_0}}}_l^l &= \int_\O \abs{\frac{H(\frac{z-z_0}{\delta})}{(\wbar{z}-\wbar{z_0})^2}}^l dm(z) \leq \int_{2\O} \abs{\frac{H(z/\delta)}{z^2}}^l dm(z) \\
&\leq \int_{2\O \setminus \frac{\delta}{2}\O} \abs{z}^{-2l} dm(z) = 2\pi \int_{\delta/2}^2 r^{1-2l}dr \\
&= \pi \frac{4^{l-1}\delta^{2(1-l)} - 4^{1-l}}{l-1} \leq {C''_l}^l \delta^{2(1-l)},
\end{split}
\end{equation}
so the claim follows by the triangle inequality of $L^l(\O)$.
\end{proof}

Note that the test function of the previous lemma is not supported compactly, so we need to take care of the boundary terms when integrating by parts. This lemma will be used for that.
\begin{lemma}
\label{boundaryInt}
Let $2<p<\infty$, $r > 0$, $z_0 \in \O$ and $n>1$. Then there is $C_r < \infty$ such that for $g \in W^{1,\frac{(2+r)p}{2+p}}(\O)$ we have
\begin{equation}
\norm{\frac{1}{2\pi}\int_{\d\O} \frac{e^{-inR}\Tr g(z')}{z-z'}z'd\sigma(z')}_{L^p(\O)} \leq C_r n^{(1-\frac{2}{(1+r)p})(1-\frac{2}{2+r})}\norm{g}_{W^{1,\frac{(2+r)p}{2p}}(\O)}.
\end{equation}
\end{lemma}
\begin{proof}
We prove the claim by interpolation. Note that $\norm{\frac{1}{z-z'}}_{L^{2/(1+r)}(\O,z)} \leq \norm{z^{-1}}_{L^{2/(1+r)}(2\O,z)} < c_r < \infty$. Thus by Minkowski's integral inequality we have
\begin{equation}
\norm{\frac{1}{2\pi} \int_{\d\O} \frac{e^{-inR}\Tr g(z')}{z-z'}z'd\sigma(z')}_{L^{\frac{2}{2+r}}(\O)} \!\!\!\!\!\!\! \leq \frac{c_r}{2\pi} \int_{\d\O} \!\!\! \abs{\Tr g(z')} d\sigma(z') \leq c_r^1 \norm{g}_{W^{1,1}(\O)}.
\end{equation}
We have the Sobolev embedding $W^{1,2+r} (\O) \subset C^{1-\frac{2}{2+r}}(\wbar{\O})$ so by \cite[thm 1.10]{vekua} and $\norm{e^{inR}}_{C^\alpha(\wbar{\O})} \leq 11 n^\alpha$ we get
\begin{multline}
\norm{\frac{1}{2\pi} \int_{\d\O} \frac{e^{-inR}\Tr g(z')}{z-z'}z'd\sigma(z')}_{L^\infty(\O)} \!\!\!\!\!\!\! \leq c_r^2 \norm{e^{inR} \Tr g}_{C^{1-\frac{2}{2+r}}(\wbar{\d\O})} \\
\leq c_r^2 \norm{e^{inR} g}_{C^{1-\frac{2}{2+r}}(\wbar{\O})} \leq c^3_r n^{1-\frac{2}{2+r}}\norm{g}_{W^{1,2+r}(\O)}.
\end{multline}
The next step is to use real interpolation $(\cdot,\cdot)_{(\theta,q)}$ with $\theta = 1-\frac{2}{(1+r)p} \in ]0,1[$ and $q^{-1} = (1-\theta)/1 + \theta/(2+r)$. Notice that
\begin{equation}
\begin{split}
&\frac{1-\theta}{\frac{2}{1+r}} + \frac{\theta}{\infty} = \frac{1}{p} \\
&\frac{1-\theta}{1} + \frac{\theta}{2+r} = \frac{2}{(1+r)p} + \frac{1}{2+r} - \frac{2}{(2+r)(1+r)p} \\
&\phantom{===}= \frac{4 + 2r + p + r p - 2}{(2+r)(1+r)p} = \frac{(1+r)(2+p)}{(2+r)(1+r)p} = \frac{2+p}{(2+r)p}.
\end{split}
\end{equation}
We use \cite[thm 6.4.5 (5)]{BL} combined with \cite[thm 6.4.2]{BL} to get the result for $\O$. These imply the claim.
\end{proof}

\begin{theorem}
\label{BIGTHM}
Let $2<p<\infty$, $\frac{1}{p*} = \frac{1}{p} + \frac{1}{2}$, 
$r > 0$. Then there is $C_{r,p} < \infty$ such that if $n>1$ then
\begin{equation}
\begin{split}
\Ca(e^{-inR} &\cdot) : C^0(W^{1,p}) \to C^0(L^p),\\
&\sup_{z_0} \norm{\Ca(e^{-inR}a)}_p \leq C_{r,p} n^{r-\frac{1}{p*}} \sup_{z_0} \norm{a}_{1,p}
\end{split}
\end{equation}
and
\begin{equation}
\begin{split}
\Ca(e^{-inR} &\cdot) : C^0(W^{1,p}) \to C^0(L^\infty),\\
&\sup_{z_0} \norm{\Ca(e^{-inR}a)}_\infty \leq C_p n^{-\frac{1}{5}} \sup_{z_0} \norm{a}_{1,p}.
\end{split}
\end{equation}
\end{theorem}
\begin{proof}
It is enough to prove the continuity of the map $z_0 \mapsto \Ca(e^{-inR(z,z_0)}a_{z_0})$ between the spaces $\O \to L^\infty$ because $L^\infty\subset L^p$. Let $\epsilon > 0$ and take $\delta >0$ such that $\norm{a_{z_0}-a_{z_0'}}_{1,p} < \frac{\epsilon}{2c_p}$ when $\abs{z_0-z_0'} < \delta$, where $c_p < \infty$ is the norm of $\Ca : L^p\to L^\infty$. Now
\begin{equation}
\begin{split}
\Big\lVert \Ca(e&^{-inR(z,z_0)}a_{z_0}) - \Ca(e^{-inR(z,z_0')}a_{z_0'}) \Big\rVert _\infty \\
&\leq c_p \norm{e^{-inR(z,z_0)}a_{z_0}(z) - e^{-inR(z,z_0')}a_{z_0'}(z)}_p \\
&\leq c_p \norm{e^{-inR(z,z_0)}a_{z_0}(z) - e^{-inR(z,z_0')}a_{z_0}(z)}_p \\
& \phantom{\leq}+ c_p \norm{e^{-inR(z,z_0')}a_{z_0}(z) - e^{-inR(z,z_0')}a_{z_0'}(z)}_p \\
&\leq c_p\norm{e^{-inR(z,z_0)}-e^{-inR(z,z_0')}}_\infty \norm{a_{z_0}}_p + c_p\norm{a_{z_0}-a_{z_0'}}_p \\
&\leq c_p \sup_{z\in\O} \norm{e^{-inR}}_{C^1(\wbar{\O},z_0)} \abs{z_0 - z_0'} \norm{a_{z_0}}_p +c_p \norm{a_{z_0} - a_{z_0'}}_p \\
&\leq 11 n c_p  \abs{z_0 - z_0'} \norm{a_{z_0}}_p + \tfrac{\epsilon}{2}< \epsilon,
\end{split}
\end{equation}
if $\abs{z_0 - z_0'} < \frac{\epsilon}{22 n c_p \norm{a_{z_0}}_p}$ and $\abs{z_0-z_0'} < \delta$. Thus it is continuous at $z_0$.

Note the following integration by parts formula: if $f \in W^{1,1}(\O)$, $z_0 \notin \supp f$ then almost everywhere
\begin{equation}
f(z) = \frac{1}{2\pi} \int_{\d\O} \frac{\Tr f (z')}{z'-z} z' d\sigma(z') + \frac{1}{\pi} \int_\O \frac{\db f(z')}{z-z'} dm(z').
\end{equation}
If $z_0 \notin \supp g$, $g\in W^{1,1}(\O)$ put $f(z) = \frac{e^{-inR}}{-2in(\wbar{z}-\wbar{z_0})}g(z)$ to get
\begin{multline}
\Ca(e^{-inR}g) = \frac{-1}{2in}\Big( e^{-inR} \frac{g}{\wbar{z}-\wbar{z_0}} - \Ca\big( e^{-inR} \db \frac{g}{\wbar{z}-\wbar{z_0}}\big) \\
+ \frac{1}{2\pi} \int_{\partial \O} \frac{e^{-inR}g(z')}{(z-z')(\wbar{z}' - \wbar{z_0})}z' d\sigma(z')\Big).
\end{multline}

The first estimate: The $h$ as in lemma \ref{testF} with $\delta = n^{-\frac{1}{2}}$. Then put $g = ha$ to get
\begin{multline}
\Ca(e^{-inR}a) =  \Ca(e^{-inR}(1-h)a) - \frac{1}{2in}\Big( e^{-inR} \frac{ha}{\wbar{z}-\wbar{z_0}} - \Ca\big( e^{-inR} \db \frac{h a}{\wbar{z}-\wbar{z_0}}\big) \\
+ \frac{1}{2\pi} \int_{\partial \O} \frac{e^{-inR}a(z')h(z')/(\wbar{z}' - \wbar{z_0})}{z-z'}z' d\sigma(z')\Big).
\end{multline}
Next take $0 < r' < \frac{4}{p}$ so small that $(1-\frac{2}{(1+r')p})(1-\frac{2}{2+r'}) - \frac{2+p}{(2+r')p} \leq r - \frac{1}{p*}$. This is possible because $\frac{2+p}{2p} = \frac{1}{p*}$, $r>0$ and the left hand side is continuous. Note that $l := \frac{(2+r')p}{2+p} \in ]1,2[$ so we may use lemma \ref{testF}. Note the fact that $\norm{AB}_{1,t} \leq c_t \norm{A}_{1,t} \norm{B}_{1,p}$ for $1\leq t \leq 2$. Keep also in mind that by lemma \ref{ONeilLemma1} we have $\Ca:L^{p*}(\O) \to L^p(\O)$.
And finally using lemma \ref{boundaryInt} and lemma \ref{testF} on the terms with $h$ we get
\begin{equation}
\begin{split}
\big\lVert&\Ca(e^{-inR}a)\big\rVert_p \leq C^1_p\Bigg( \norm{(1-h)a}_{p*} + n^{-1} \Big(\norm{\frac{ha}{\wbar{z}-\wbar{z_0}}}_p + \norm{\frac{ha}{\wbar{z}-\wbar{z_0}}}_{1,p*} \\
&\phantom{\leq} + n^{(1-\frac{2}{(1+r')p})(1-\frac{2}{2+r'})} \norm{\frac{h}{\wbar{z}-\wbar{z_0}} a }_{1,l} \Big)\Bigg) \\
&\leq C^2_{r,p}\Big( \norm{1-h}_{p*} \norm{a}_\infty + n^{-1}\Big(\norm{\frac{h}{\wbar{z}-\wbar{z_0}}}_p \norm{a}_\infty + \norm{\frac{h}{\wbar{z}-\wbar{z_0}}}_{1,p*} \norm{a}_{1,p} \\
&\phantom{\leq} + n^{(1-\frac{2}{(1+r')p})(1-\frac{2}{2+r'})} \norm{\frac{h}{\wbar{z}-\wbar{z_0}}}_{1,l} \norm{a}_{1,p}\Big)\Big) \\
&\leq C^3_{r,p}\Big( \norm{1-h}_{p*} + n^{-1}\Big(\norm{\frac{h}{\wbar{z}-\wbar{z_0}}}_{1,p*}  \!\!\!\!\!\!+ n^{(1-\frac{2}{(1+r')p})(1-\frac{2}{2+r'})} \norm{\frac{h}{\wbar{z}-\wbar{z_0}}}_{1,l}\Big)\Big) \norm{a}_{1,p} \\
&\leq C^4_{r,p}\big(\delta^{2/p*} + n^{-1}(\delta^{2(1/p*-1)} + n^{(1-\frac{2}{(1+r')p})(1-\frac{2}{2+r'})}\delta^{2(1/l-1)})\big) \norm{a}_{1,p} \\
&= C^4_{r,p} \big(n^{-1/p*} + n^{-1}(n^{1-1/p*} + n^{(1-\frac{2}{(1+r')p})(1-\frac{2}{2+r'}) + 1 - \frac{2+p}{(2+r')p}} )\big) \norm{a}_{1,p} \\
&\leq C^5_{r,p} ( n^{-1/p*} + n^{r - 1/p*}) \norm{a}_{1,p} \leq 2C_{r,p}^6 n^{r-1/p*} \norm{a}_{1,p}.
\end{split}
\end{equation}
The claim follows since $C^6_{r,p}$ does not depend on $z_0$.

\bigskip
The second estimate: For $\delta \in ]0,1[$ and $z_0 \in \O$ take $h\in C^\infty_0(\wbar{\O})$ such that it is continuous with respect to $z_0$ and
\begin{enumerate}
\item $0\leq h \leq 1$,
\item $h(z) = 0 \Leftrightarrow \abs{z-z_0} \leq \delta/2 \text{ or } \abs{z} \geq 1 - \delta/2$,
\item $h(z) = 1 \Leftrightarrow \abs{z-z_0} \geq \delta \text{ and } \abs{z} \leq 1 - \delta$,
\item $m(\supp(1-h)) \leq 2\pi \delta$,
\item $\sup_{z_0} \norm{\frac{h}{\wbar{z}-\wbar{z_0}}}_{C^0} \leq c \delta^{-1}$, $\sup_{z_0} \norm{\frac{h}{\wbar{z}-\wbar{z_0}}}_{C^1} \leq c \delta^{-2}$.
\end{enumerate}
This kind of test function exists by the construction of \cite[5.3.2]{lis}. Now integrate $g = ha$ by parts to get
\begin{equation}
\Ca(e^{-inR}a) = \Ca\big(e^{-inR}(1-h)a\big) - \frac{1}{2in}\Big( e^{-inR} \frac{ha}{\wbar{z}-\wbar{z_0}} - \Ca\big( e^{-inR} \db \frac{ha}{\wbar{z}-\wbar{z_0}} \big)\Big).
\end{equation}
We don't need a very sharp bound here. Use the fact that $\Ca : L^{(2,1)} \to L^\infty$ on the first term and $\Ca:L^p\to L^\infty$ on the second term with $\Ca$. Moreover notice that $\norm{e^{inR}(1-h)a}_{(2,1)} \leq c \norm{\chi_{\supp 1-h}}_{(2,1)} \norm{a}_\infty \leq c' m(\supp 1-h)^{1/2} \norm{a}_{1,p}$. Thus
\begin{equation}
\begin{split}
\norm{\Ca(e^{-inR}a)}_\infty &\leq C'_p\Big( \norm{\chi_{\supp1-h}}_{(2,1)} \norm{a}_\infty + n^{-1} \norm{\frac{h}{\wbar{z}-\wbar{z_0}}}_{C^0} \norm{a}_\infty \\
&\phantom{\leq} + n^{-1} \norm{\frac{ha}{\wbar{z}-\wbar{z_0}}}_{1,p} \Big)\\
&\leq C''_p ( \delta^{1/2} + n^{-1}\delta^{-1} + n^{-1}\norm{\frac{h}{\wbar{z}-\wbar{z_0}}}_{C^1}) \norm{a}_{1,p}\\
&\leq C'''_p ( \delta^{1/2} + n^{-1}\delta^{-2}) \norm{a}_{1,p}.
\end{split}
\end{equation}
Then choose $\delta = n^{-\frac{1}{2+1/2}}$ to get $\delta^{1/2} = n^{-1}\delta^{-2} = n^{-1/5}$. The claim follows because the coefficients do not depend on $z_0$.
\end{proof}

Next we will start to use interpolation more seriously. We use notations and definitions from \cite{BL}.

\begin{definition}
\label{functorReq}
Let $F_\theta$ denote any exact interpolation functor of exponent $\theta \in [0,1]$ in the category of Banach spaces such that $F_0(A,B) = A$, $F_1(A,B) = B$ and which satisfies multilinear interpolation. That is if $(A_0^{(j)},A_1^{(j)})$, $j=1,\ldots, m$, and $(B_0,B_1)$ are compatible Banach couples and $T$ is any multilinear bounded mapping satisfying
\begin{equation}
\begin{cases}
T:A_0^{(1)} \oplus \ldots \oplus A_0^{(m)} \to B_0 \text{ with norm } M_0,\\
T:A_1^{(1)} \oplus \ldots \oplus A_1^{(m)} \to B_1 \text{ with norm } M_1,
\end{cases}
\end{equation}
then
\begin{equation}
T:F_\theta(A_0^{(1)}, A_1^{(1)}) \oplus \ldots \oplus F_\theta(A_0^{(m)}, A_1^{(m)}) \to F_\theta(B_0, B_1),
\end{equation}
with norm at most $M_0^{1-\theta}M_1^\theta$.
\end{definition}

\begin{lemma}
\label{functorExist}
The complex interpolation $(\cdot,\cdot)_{[\theta]}$, $0 < \theta < 1$, the real interpolation $(\cdot,\cdot)_{\theta,1}$, $0 < \theta < 1$ and the trivial ones $(A,B)_0 = A$, $(A,B)_1 = B$ satisfy the requirements in definition \ref{functorReq}.
\end{lemma}
\begin{proof}
The trivial ones clearly satisfy the claims. Complex interpolation satisfies them by \cite[thms 4.1.2, 4.4.1]{BL}. Real interpolation with $q=1$ satisfies them by \cite[thms 3.1.2, 3.2.2, 3.3.1, 3.4.2, ex 3.5(a)]{BL}.
\end{proof}

\begin{remark}
To conserve space we write $X_\theta = F_\theta\big(C^0(\wbar{\O},L^p), C^0(\wbar{\O},W^{1,p})\big)$ and $A_\theta = F_\theta(L^p,W^{1,p})$. Also denote $X'_\theta = F_\theta\big(C^0(\wbar{\O},L^\infty), C^0(\wbar{\O},W^{1,p})\big)$. These are well defined, because $W^{1,p}(\O), L^\infty(\O), L^p(\O) \subset L^1(\O)$, which is a Hausdorff space. It is assumed that $z$ is the variable of the Sobolev space and $z_0$ the one of the continous functions.
\end{remark}

\begin{lemma}
\label{noZ0}
Let $0\leq \theta \leq 1$. If $f\in A_\theta$ then $\lVert\tilde{f}\rVert_{X_\theta} = \norm{f}_{A_\theta}$, where $\tilde{f}(z_0) = f$ for all $z_0 \in \wbar{\O}$.
\end{lemma}
\begin{proof}
For any Banach space $B$ consider the operators $I:B \to C^0(\wbar{\O},B)$, $Ig(z_0) = g$ for all $z_0 \in \wbar{\O}$, and $P:C^0(\wbar{\O},B) \to B$, $Pf = f(0)$. Then 
\begin{equation}
\begin{cases}
I: L^p(\O) \to C^0(\wbar{\O}, L^p(\O)), \quad \norm{Ig}_{C^0(L^p)} = \norm{g}_{L^p}\\
I: W^{1,p}(\O) \to C^0(\wbar{\O}, W^{1,p}(\O)), \quad \norm{Ig}_{C^0(W^{1,p})} = \norm{g}_{W^{1,p}}
\end{cases}
\end{equation}
so by interpolating with $F_\theta$ we have $I:A_\theta \to X_\theta$, $\norm{Ig}_{X_\theta} \leq \norm{g}_{A_\theta}$. Similarly we get $P:X_\theta \to A_\theta$, $\norm{Pf}_{A_\theta} \leq \norm{f}_{X_\theta}$. But $\tilde{f} = If$ and $P\tilde{f} = f$, so 
\begin{equation}
\lVert \tilde{f} \rVert_{X_\theta} = \norm{If}_{X_\theta} \leq \norm{f}_{A_\theta} = \lVert P\tilde{f} \rVert_{A_\theta} \leq \lVert \tilde{f} \rVert_{X_\theta}.
\end{equation}
\end{proof}

\begin{remark}
Using this lemma we can make sense of expressions like $q + f$, $qf$, etc\ldots when $q \in A_\theta, f\in X'_\theta \cup X_\theta$. We won't usually explicitly write out the operators $I$ and $P$.
\end{remark}

\bigskip
\begin{corollary}[to thm \ref{BIGTHM}]
\label{corollary1}
Let $n>1$, $2<p<\infty$, $\frac{1}{p*} = \frac{1}{2} + \frac{1}{p}$, $r > 0$, $\theta \in [0,1]$. Then there exists $C_{r,p} < \infty$ such that if $q\in A_\theta(\O)$ we have
\begin{align}
&\norm{\Ca\big(e^{-inR}\Cab(e^{inR}a)\big)}_{X_\theta} \leq C_{r,p} n^{r-\frac{1}{p*}}\norm{a}_{X_\theta},\\
&\norm{\Ca\big(e^{-inR}\Cab(e^{inR}qf)\big)}_{X'_\theta} \leq C_{r,p} n^{(r-\frac{1}{p*})\theta - \frac{1}{5}(1-\theta)}\norm{q}_{A_\theta} \norm{f}_{X'_\theta},
\end{align}
with corresponding mapping properties.
\end{corollary}
\begin{proof}
It is enough to prove the limiting cases and the rest will follow from the definition of $F_\theta$. We use theorem \ref{BIGTHM} and the facts that $\norm{\Ca f}_{1,p} \leq c_p \norm{f}_p$. We get the mapping properties and the following estimates uniformly in $z_0$:
\begin{equation}
\begin{split}
&\norm{\Ca\big(e^{-inR}\Cab(e^{inR}a)\big)}_p \leq C_{r,p} n^{r-\frac{1}{p*}} \norm{\Cab(e^{inR}a)}_{1,p} \leq c_p C_{r,p} n^{r-\frac{1}{p*}} \norm{a}_p,\\
&\norm{\Ca\big(e^{-inR}\Cab(e^{inR}a)\big)}_{1,p} \leq c_p \norm{\Cab(e^{inR}a)}_p \leq c_p C_{r,p} n^{r-\frac{1}{p*}} \norm{a}_{1,p},
\end{split}
\end{equation}
so the first claim follows.

For the second claim we use the second part of theorem \ref{BIGTHM}:
\begin{equation}
\begin{split}
\big\lVert\Ca\big(&e^{-inR}\Cab(e^{inR}qf)\big)\big\rVert_\infty \leq C_p n^{-\frac{1}{5}} \norm{\Cab(e^{inR}qf)}_{1,p} \leq c_pC_p n^{-\frac{1}{5}} \norm{qf}_{p} \\
&\leq c_p C_p n^{-\frac{1}{5}} \norm{q}_p \norm{f}_\infty, \\
\big\lVert\Ca\big(&e^{-inR}\Cab(e^{inR}qf)\big)\big\rVert_{1,p} \leq c_p \norm{\Cab(e^{inR}qf)}_p \leq c_p C_{r,p} n^{r-\frac{1}{p*}} \norm{qf}_{1,p} \\
&\leq c_pC_{r,p}A_p n^{r-\frac{1}{p*}} \norm{q}_{1,p} \norm{f}_{1,p},
\end{split}
\end{equation}
so the second claim follows, because the coefficients do not depend on $z_0$.
\end{proof}

The idea to continue is to take solutions $f_{z_0,n}$ from $X'_\theta$ by using the second estimate in corollary \ref{corollary1}. This allows us to multiply by $f$ because $X'_\theta$ is a multiplier space for $A_\theta$ (more exactly, for $I A_\theta \subset X_\theta$, see lemma \ref{noZ0} and the remark after it). After that the first estimate gives $\norm{f-1} \leq n^{r-\frac{1}{p*}} \norm{q}_{A_\theta}$. Basically this is a sort of boot-strapping argument. The following lemma is needed for the boot-strapping.

\smallskip
In particular here we should use $I A_\theta$, but we identify it with $A_\theta$.
\begin{lemma}
\label{multiplierLemma}
Let $2<p<\infty$. Then there is $C_p < \infty$ such that for all $\theta \in [0,1]$, $f\in A_\theta$, $g\in X'_\theta$ we have $fg \in X_\theta$ with $\norm{fg}_{X_\theta} \leq C_p \norm{f}_{A_\theta} \norm{g}_{X'_\theta}$.
\end{lemma}
\begin{proof}
This follow by multilinear interpolation and the fact that $W^{1,p}$ is a Banach algebra:
\begin{equation}
\begin{split}
\sup_{z_0} \norm{fg}_p &\leq \norm{f}_p \sup_{z_0}\norm{g}_\infty \\
\sup_{z_0} \norm{fg}_{1,p} &\leq C_p \norm{f}_{1,p} \sup_{z_0} \norm{g}_{1,p}.
\end{split}
\end{equation}
\end{proof}

Next is the big theorem, which shows the existence of suitable solutions and gives the behaviour of the remainder terms.
\begin{definition}
\label{n0def}
By $n_0(r,p,\theta,M)$ we denote the number
\begin{equation}
\max\left(1, (C_{r,p}M)^{-1\left/\left((r-\frac{1}{p*})\theta - \frac{1}{5}(1-\theta)\right)\right.} \right),
\end{equation}
which grows with $M$ if $r < \frac{1}{p*}$.
\end{definition}

\begin{theorem}
\label{solEx}
Let $2<p<\infty$, $\frac{1}{p*} = \frac{1}{2} + \frac{1}{p}$, $0<r < \frac{1}{p*}$, $\theta \in [0,1]$, $q\in A_\theta$, $n \geq n_0(r,p,\theta,\norm{q}_{A_\theta})$. Then there is a unique $f_n \in X'_\theta$ such that for all $z_0$
\begin{equation}
f_n = 1 - \tfrac{1}{4} \Ca\big(e^{-inR}\Cab(e^{inR}qf_n)\big).
\end{equation}
Moreover we have $f_n \in X_\theta$ and
\begin{equation}
\norm{f_n-1}_{X_\theta} \leq C_{r,p} n^{r-\frac{1}{p*}} \norm{q}_{A_\theta} \text{ and } \sup_{z_0} \norm{f_n}_{1,p} \leq 2 + C'_p \norm{q}_{A_\theta}.
\end{equation}
\end{theorem}
\begin{proof}
Define $T_n$ by $f \mapsto 1 - \tfrac{1}{4} \Ca\big(e^{-inR}\Cab(e^{inR}qf)\big)$. By corollary \ref{corollary1} we have $T_n : X'_\theta \to X'_\theta$ and get the norm estimate
\begin{multline}
\norm{T_n f - T_n f'}_{X'_\theta} = \tfrac{1}{4}\norm{\Ca\big(e^{-inR}\Cab(e^{inR}q(f-f'))\big)}_{X'_\theta}\\
\leq \tfrac{1}{4}c_{r,p} n^{(r-\frac{1}{p*})\theta - \frac{1}{5}(1-\theta)} \norm{q}_{A_\theta} \norm{f-f'}_{X'_\theta} \leq \tfrac{1}{2} \norm{f-f'}_{X'_\theta},
\end{multline}
because $n \geq n_0(r,p,\theta,\norm{q}_{A_\theta}) < \infty$. Thus $T_n$ is a contraction in the Banach space $X'_\theta$ and so has a unique fixed point $f_n$. there.

To prove the second claim do the same reasoning as in the previous formula and so
\begin{equation}
\norm{f_n}_{X'_\theta} \leq \norm{1}_{X'_\theta} + \tfrac{1}{4} \norm{\Ca\big(e^{-inR}\Cab(e^{inR}qf_n)\big)}_{X'_\theta} \leq c'_p + \tfrac{1}{2}\norm{f_n}_{X'_\theta},
\end{equation}
because $n \geq n_0(r,p,\theta,\norm{q}_{A_\theta})$. Thus $\norm{f_n}_{X'_\theta} \leq 2 c'_p$. Now by the first norm estimate of corollary \ref{corollary1} and the multiplier lemma \ref{multiplierLemma} we get
\begin{multline}
\norm{f_n-1}_{X_\theta} = \tfrac{1}{4}\norm{\Ca\big(e^{-inR}\Cab(e^{inR}qf_n)\big)}_{X_\theta} \\
\leq \tfrac{1}{4} c_{r,p} n^{r-\frac{1}{p*}} \norm{qf_n}_{X_\theta} \leq C_{r,p} n^{r-\frac{1}{p*}} \norm{q}_{A_\theta}.
\end{multline}

The last claim follows from the well-known fact that $\Ca: L^p \to W^{1,p}$ and the embedding $X_\theta \subset C^0(L^p)$:
\begin{multline}
\sup_{z_0} \norm{f_n}_{1,p} \leq \pi^{1/p} + \sup_{z_0} \tfrac{1}{4}\norm{\Ca\big(e^{-inR}\Cab(e^{inR}qf_n)\big)}_{1,p} \\
\leq 2 + \tfrac{1}{4}c_p^2 \sup_{z_0} \norm{qf_n}_p \leq 2+\tfrac{1}{4} c_p^2 \norm{qf_n}_{X_\theta} \leq 2 + C'_p \norm{q}_{A_\theta}.
\end{multline}
\end{proof}

Next we handle the error term integral.

\begin{theorem}
\label{errorTermIntegral}
Let $2<p<\infty$. Then there exists $C_p < \infty$ such that if $n>0$, $\theta \in [0,1]$, $Q\in A_\theta$ and $r_{z_0,n} \in X_\theta$ we have
\begin{equation}
\norm{\int_\O \frac{2n}{\pi}e^{inR}Q(z)r_{z_0,n}(z)\,dm(z) }_{L^2(\O,z_0)} \leq C_p n^{1-\theta} \norm{Q}_{A_\theta} \norm{r_{z_0,n}}_{X_\theta}.
\end{equation}
\end{theorem}
\begin{proof}
By \cite[Thm 5.2.6]{lis} we have a $C'_p<\infty$ such that
\begin{equation}
\norm{\int_\O \frac{2n}{\pi}e^{inR}Qr_{z_0,n}\,dm(z) }_{L^2(\O,z_0)} \leq C'_p \norm{Q}_{1,p} \sup_{z_0}\norm{r_{z_0,n}}_{1,p}.
\end{equation}
Because $m(\O) = \pi < \infty$, $p>2$ and H\"older's inequality we get
\begin{multline}
\norm{\int_\O \frac{2n}{\pi}e^{inR}Qr_{z_0,n}\,dm(z) }_{L^2(\O,z_0)} \leq \pi^{1/2} \norm{\int_\O \frac{2n}{\pi}e^{inR}Qr_{z_0,n}\,dm(z) }_{L^\infty(\O,z_0)} \\
\leq \tfrac{2}{\pi^{1/2}} n \norm{Q}_2 \sup_{z_0} \norm{r_{z_0,n}}_2 \leq \tfrac{2}{\pi^{1/2}}(\pi^{1/2-1/p})^2 n \norm{Q}_p \sup_{z_0} \norm{r_{z_0,n}}_p.
\end{multline}
Because $F_\theta$ satisfies multilinear interpolation we get the result.
\end{proof}

\section{Well-posedness and the inverse problem}
Here we define the Dirichlet-Neumann operator, prove an orthogonality formula and define what does it mean that the direct problem is well-posed. In this section we denote $H^s = W^{s,2}$.
\begin{definition}
Let $q\in \mathscr{D}'(\O)$. Then \emph{the direct problem is well-posed} if there is $C<\infty$ such that for any $f \in H^{1/2}(\d \O)$ we have
\begin{enumerate}
\item there is $u\in H^1(\O)$ such that $\Delta u + q u = 0$, $\Tr u = f$,
\item this $u$ is unique
\item $u$ depends continuously on $f$: $\norm{u}_{H^1(\O)} \leq C \norm{f}_{H^{1/2}(\d \O)}$.
\end{enumerate}
\end{definition}

\begin{definition}
Let $q\in L^a(\O)$, $a > 1$, be such that the direct problem is well-posed. Then we define the \emph{Dirichlet-Neumann operator $\Lambda_q$} as follows. For $f \in H^{1/2}(\d\O)$ we define $\Lambda_q f \in H^{-1/2}(\d\O)$ by
\begin{equation}
\label{DNmap}
(\Lambda_q f , g) = \int_\O (-\nabla u \cdot \nabla v + q u v )dm, \quad g \in H^{1/2}(\d\O),
\end{equation}
for any $u,v \in H^1(\O)$ such that $\Tr u = f$, $\Tr v = g$ and $\Delta u + q u = 0$.
\end{definition}

\begin{lemma}
\label{DNok}
The Dirichlet-Neumann operator is well defined and $\Lambda_q : f \mapsto \Lambda_q f $ is a continuous linear operator mapping $H^{1/2}(\d\O) \to H^{-1/2}(\d\O)$ which satisfies
\begin{equation}
(\Lambda_q f, g) = (\Lambda_q g, f), \quad f,g \in H^{1/2}(\d\O).
\end{equation}
\end{lemma}
\begin{proof}
By the well-posedness of the direct problem $u$ is unique on the right-hand side of \eqref{DNmap}. Assume that $v,v' \in H^1(\O)$ satisfy $\Tr v = g = \Tr v'$. Now $v-v' \in H^1_0(\O)$ and because $u$ is a solution to the Schr\"odinger equation we have
\begin{equation}
\int_\O (-\nabla u \cdot \nabla (v-v') + q u (v-v')) dm = 0,
\end{equation}
which implies that all choices of $v$ give the same value for the right-hand side of \eqref{DNmap}.

Note that $H^{-1/2}(\d\O) = \big( H^{1/2}(\d\O) \big)^\ast$. Thus to prove the mapping properties of $\Lambda_q$ it is enough to prove that for a fixed $f\in H^{1/2}(\d\O)$ we have
\begin{equation}
\abs{\int_\O (-\nabla u \cdot \nabla v + quv) dm} \leq C_{\O,a,q} \norm{f}_{H^{1/2}(\d\O)} \norm{\Tr v}_{H^{1/2}(\d\O)}.
\end{equation}
Let $R:H^{1/2}(\d\O) \to H^1(\O)$ be a bounded right inverse to $\Tr$. Note the Sobolev embedding $H^1(\O) \subset L^{\frac{2a}{a-1}}(\O)$ because $\frac{2a}{a-1} < \infty$. Denote $\frac{1}{a}+\frac{1}{a'} = 1$, so by H\"older's inequality, Sobolev embedding and the third condition of the well-posedness of $q$ we get
\begin{equation}
\begin{split}
\Big|&\int_\O (-\nabla u \cdot \nabla v + quv)dm\Big| = \abs{\int_\O (-\nabla u \cdot \nabla R(g) + quR(g)) dm} \\
&\leq \norm{\nabla u}_2 \norm{\nabla R(g)}_2 + \norm{quR(g)}_1 \leq \norm{u}_{H^1} \norm{R(g)}_{H^1} + \norm{q}_a \norm{uR(g)}_{a'} \\
&\leq \norm{u}_{H^1} \norm{R(g)}_{H^1} + \norm{q}_a \norm{u}_{\frac{2a}{a-1}} \norm{R(g)}_{\frac{2a}{a-1}} \\
&\leq C_{\O,a}(1+\norm{q}_a) \norm{u}_{H^1} \norm{R(g)}_{H^1} \\
&\leq C_{\O,a,q} \norm{f}_{H^{1/2}(\d\O)} \norm{g}_{H^{1/2}(\d\O)}.
\end{split}
\end{equation}
To prove the last formula let $f,g \in H^{1/2}(\d\O)$ and $F,G \in H^1(\O)$ be the corresponding solutions to the well-posed direct problem. Now
\begin{equation}
(\Lambda_q f, g) = \int_\O -\nabla F \cdot \nabla G + qFG dm = \int_\O -\nabla G \cdot \nabla F + q GF dm = (\Lambda_q g,f).
\end{equation}
\end{proof}

\begin{theorem}
\label{DNint}
Let $q_1,q_2 \in L^a(\O)$, $a>1$, be such that the direct problem is well-posed. Let $u_1,u_2 \in H^{1}(\O)$ satisfy $\Delta u_j + q_j u_j = 0$. Then
\begin{equation}
\int_\O u_1(q_1-q_2)u_2 dm = \big( (\Lambda_{q_1} - \Lambda_{q_2})\Tr u_1, \Tr u_2\big).
\end{equation}
\end{theorem}
\begin{proof}
Add $-\nabla u_2 \cdot \nabla u_1 + \nabla u_2 \cdot \nabla u_1$ to the left side to get by definition
\begin{multline}
\int_\O u_1(q_1 - q_2)u_2 dm = \int_\O (-\nabla u_1 \cdot \nabla u_2 + q_1 u_1 u_2) dm \\
- \int_\O (-\nabla u_2 \cdot \nabla u_1 + q_2 u_2 u_1) dm = (\Lambda_{q_1} \Tr u_1, \Tr u_2) - (\Lambda_{q_2} \Tr u_2, \Tr u_1).
\end{multline}
The claim follows by lemma \ref{DNok} because $(\Lambda_{q_2} \Tr u_2, \Tr u_1) = (\Lambda_{q_2} \Tr u_1, \Tr u_2)$.
\end{proof}

\section{The proof}
First two technical lemmas:
\begin{lemma}
\label{polynomialLogarithm}
Let $0<x<e^{-1}$, $\alpha > 0$, $\beta \in \R$. Then $x^\alpha \leq (\ln \frac{1}{x})^{-\beta}$ if $\alpha \geq \beta e^{-1}$.
\end{lemma}
\begin{proof}
The cases $\beta \leq 0$ are clear because $\ln \frac{1}{x} \geq 1$. Assume $\beta > 0$. It is easily seen that $x^\alpha \leq (\ln \frac{1}{x})^{-\beta} \Leftrightarrow x^{\alpha/\beta}\ln \frac{1}{x} \leq 1$. Write $f(x) = x^{\alpha/\beta} \ln \frac{1}{x}$. Now
\begin{multline}
	f'(x) = \tfrac{\alpha}{\beta}x^{\alpha/\beta-1}\ln \tfrac{1}{x} + x^{\alpha/\beta}(-\tfrac{1}{x^2})x = x^{\alpha/\beta-1}(\tfrac{\alpha}{\beta}\ln \tfrac{1}{x} - 1) \geq 0\\
	\Leftrightarrow \ln \tfrac{1}{x} \geq \tfrac{\beta}{\alpha} \Leftrightarrow x \leq e^{-\beta/\alpha}.
\end{multline}
So the maximum of $f$ is at $x = e^{-\beta/\alpha}$.
\begin{equation}
	f(e^{-\beta/\alpha}) = e^{-1} \cdot \tfrac{\beta}{\alpha} \leq 1 \Leftrightarrow \alpha \geq \beta e^{-1}.
\end{equation}
\end{proof}

\begin{lemma}
\label{uNorm}
There is $C<\infty$ such that if $n \in \R$, $f_1,f_2 \in C^0(\wbar{\O},W^{1,2}(\O))$ and $u_1(z) = e^{in(z-z_0)^2} f_1(z)$, $u_2 = e^{in(\wbar{z}-\wbar{z_0})^2} f_2 (z)$ then
\begin{equation}
\sup_{z_0} \norm{u_j}_{W^{1,2}(\O)} \leq C e^{5n} \sup_{z_0} \norm{f_j}_{W^{1,2}(\O)}.
\end{equation}
\end{lemma}
\begin{proof}
This is a more or less direct calculation using the elementary facts that $\abs{z-z_0}, \abs{\wbar{z}-\wbar{z_0}} \leq 2$, $n \leq e^n$, $\abs{\d e^{in(z-z_0)^2}}, \abs{\db e^{in(\wbar{z}-\wbar{z_0})^2}} \leq 4ne^{4n}$ and $\sqrt{a^2 + b^2 + c^2} \leq a + b + c$.
\end{proof}

\begin{theorem}
Let $M>0$, $\varepsilon > 0$, $a > 2$. Then there are positive real numbers $C_{M,\varepsilon,a}, C'_{M,\varepsilon,a}$ such that if $q_1,q_2 \in W^{\varepsilon,a}(\Omega)$ are such that the direct problem is well posed and
\begin{equation}
\begin{cases}
\norm{q_j}_{W^{\varepsilon,a}} \leq M \\
\norm{\Lambda_{q_1} - \Lambda_{q_2}}_{H^{1/2}(\d\O) \to H^{-1/2}(\d\O)} \leq C_{M,\varepsilon,a}
\end{cases}
\end{equation}
then
\begin{equation}
\norm{q_1 - q_2}_{L^2(\O)} \leq C'_{M,\varepsilon,a} \big( \ln \norm{\Lambda_{q_1} - \Lambda_{q_2}}^{-1} \big)^{-\min(4\varepsilon,1)/8},
\end{equation}
so we have uniqueness and stability for the cases $\epsilon > 0$.
\end{theorem}
\begin{proof}
Choose $\theta = \min(\varepsilon, \frac{1}{4})$, $r = \frac{\theta}{4}$ and $p = \min(a,\frac{4}{2-\theta})$. Now $0 < \theta < \frac{1}{2}$ and $2<p< \frac{2}{1-\theta}$. So, $0 < r < \frac{1}{p*}$ and $1 - \theta + r - \frac{1}{p*} \leq -\frac{\theta}{2}$, which will be used in formula \eqref{simplifyExp} to simplify notations. Moreover we have $q_j \in W^{\theta,p}$ and there is a constant $c_{\varepsilon,a} < \infty$ such that $\norm{q_j}_{W^{\theta, p}} \leq c_{\varepsilon,a} M$.

The reason for doing this was that we are going to look for solutions to the Schr\"odinger equations in $L^p$ based spaces. When $p$ is as close to $2$ as possible we will get as much decay for the remainder terms as possible. The decay is almost $n^{-\frac{1}{p*}}$, which is almost $n^{-1}$ when $p \approx 2$.

\smallskip
Denote by $F_\theta$ the complex interpolation $(\cdot,\cdot)_{[\theta]}$ as defined in the book \cite{BL}. If $\theta = 1$ then let $F_1 (A,B) = B$. Remember that we write 
\begin{equation}
\begin{split}
A_\theta &= F_\theta\big(L^p(\O),W^{1,p}(\O)\big) = W^{\theta,p}(\O),\\
X_\theta &= F_\theta\big(C^0(\wbar{\O},L^p(\O)), C^0(\wbar{\O},W^{1,p}(\O))\big).
\end{split}
\end{equation}
Now by lemma \ref{functorExist} we may use the theorems and lemmas of the preceding sections.

Denote $Q = q_1 - q_2$ and $R = (z-z_0)^2 + (\wbar{z}-\wbar{z_0})^2$ for $z,z_0 \in \C$. Remember that $z_0$ is the variable of the continuous function in $X_\theta$. Assume that $n>0$. By the triangle inequality
\begin{equation}
\norm{q_1 - q_2}_{L^2(\O)} \leq \norm{ Q - \int_\O \tfrac{2n}{\pi} e^{inR}Q dm(z)}_{L^2(\O,z_0)} \!\!\!\!\!\!+ \norm{\tfrac{2n}{\pi} \int_\O e^{inR}Q dm(z) }_{L^2(\O,z_0)}.
\end{equation}
Next we will use stationary phase. Let $E:W^{\theta,2}(\O) \to W^{\theta,2}(\C)$ be an extension operator. This exists by \cite[3.3.4]{triebel1} because $0<\theta<\frac{1}{2}$. Let $\chi_\O$ be the characteristic function of the unit disc. Then remembering when characteristic functions are multipliers (\cite[3.3.2]{triebel1}) and the embedding $W^{\theta,p} \subset W^{\theta,2}$ (\cite[3.3.1]{triebel1}) we get by theorem \ref{stationaryPhase}
\begin{equation}
\begin{split}
\bigg\lVert& Q - \frac{2n}{\pi} \int_\O e^{inR} Q dm(z)\bigg\rVert_{L^2(\O,z_0)} \!\!\!\!\!\! = \norm{ \chi_\O E Q - \frac{2n}{\pi} \int_\C e^{inR} \chi_\O E Q dm(z)}_{L^2(\C,z_0)}\\
&\leq C_\theta n^{-\theta/2} \norm{\chi_\O E Q}_{W^{\theta,2}(\C)} \leq C'_\theta n^{-\theta/2} \norm{E Q}_{W^{\theta,2}(\C)} \leq C''_\theta n^{-\theta/2} \norm{Q}_{W^{\theta,2}(\O)} \\
&\leq C'''_{\theta,p} n^{-\theta/2} \norm{Q}_{W^{\theta,p}(\O)} \leq C_{\theta,p,M} \, n^{-\theta/2}.
\end{split}
\end{equation}
Next the second term. Take $n_0 = n_0(r,\theta,p,M)$ as in definition \ref{n0def}. Then take $n = \frac{1}{22} \ln \norm{\Lambda_{q_1} - \Lambda_{q_2}}^{-1}$. We may choose $C_{M,\varepsilon,a}$ in the a-priori assumptions so small and positive that $n \geq n_0(r,\theta,p,M)$: Take $C_{M,\varepsilon,a}>0$ to be a solution to $\frac{1}{22}\ln x^{-1} \geq n_0(r,\theta,p,M)$ such that $C_{M,\varepsilon,a} < e^{-1}$. Remember that $r$ and $\theta$ are functions of $\varepsilon$, and $p$ is a function of $a$ and $\varepsilon$.

Because $n_0$ grows with $M$, by theorem \ref{solEx} (the sign of $i$ does not matter) there exists $f^{(1)}, f^{(2)} \in X_\theta$ such that for all $z_0 \in \O$ we have
\begin{equation}
\begin{cases}
f^{(1)} = 1 - \tfrac{1}{4} \Ca\big( e^{-inR} \Cab (e^{inR} q_1 f^{(1)})\big),\\
f^{(2)} = 1 - \tfrac{1}{4} \Ca\big( e^{inR} \Cab (e^{-inR} q_2 f^{(2)})\big),
\end{cases}
\end{equation}
and
\begin{equation}
\begin{cases}
\norm{f^{(j)} - 1}_{X_\theta} \leq c_{r,p,M} n^{r-\frac{1}{p*}},\\
\sup_{z_0} \norm{f^{(j)}}_{W^{1,p}(\O)} \leq c_{p,M} < \infty.
\end{cases}
\end{equation}
Denote
\begin{equation}
\begin{cases}
u^{(1)}_{z_0}(z) = e^{in(z-z_0)^2}f^{(1)}(z_0,z),\\
u^{(2)}_{z_0}(z) = e^{in(\wbar{z}-\wbar{z_0})^2}f^{(2)}(z_0,z).
\end{cases}
\end{equation}
Now they satisfy $u^{(j)}_{z_0} \in C^0(\wbar{\O},W^{1,p}(\O))$ and $\Delta u^{(j)}_{z_0} + q_j u^{(j)}_{z_0} = 0$ for all $z_0$. Moreover by lemma \ref{uNorm} and the embedding $W^{1,p} \subset W^{1,2}$ we have
\begin{equation}
\label{solNorms}
\sup_{z_0} \norm{u^{(j)}_{z_0}}_{W^{1,2}(\O)} \leq c'_{p,M} e^{5n}.
\end{equation}
Now by the triangle inequality
\begin{multline}
\norm{\frac{2n}{\pi} \int_\O e^{inR} Q dm(z)}_{L^2(\O,z_0)} \leq \norm{\frac{2n}{\pi} \int_\O u^{(1)}_{z_0}(q_1 - q_2)u^{(2)}_{z_0} dm(z)}_{L^2(\O,z_0)} \\
+ \norm{\frac{2n}{\pi} \int_\O e^{inR}Q(f^{(1)} f^{(2)} - 1) dm(z)}_{L^2(\O,z_0)}.
\end{multline}
For the first term here we use theorem \ref{DNint}, formula \eqref{solNorms} and the fact that $\Tr : H^1(\O) \to H^{1/2}(\d\O)$ to get
\begin{equation}
\begin{split}
\bigg\lVert& \frac{2n}{\pi} \int_\O u^{(1)}_{z_0}(q_1 - q_2)u^{(2)}_{z_0} dm(z) \bigg\rVert_{L^2} = \norm{\frac{2n}{\pi} \big((\Lambda_{q_1}-\Lambda_{q_2}) \Tr u^{(1)}_{z_0} , \Tr u^{(2)}_{z_0} \big)}_{L^2} \\
&\leq \frac{2n}{\pi} \norm{\Lambda_{q_1} - \Lambda_{q_2}}_{H^{1/2}(\d\O) \to H^{-1/2}(\d\O)} \norm{ \norm{\Tr u^{(1)}_{z_0}}_{H^{1/2}(\d\O)} \norm{\Tr u^{(2)}_{z_0}}_{H^{1/2}(\d\O)} }_{L^2} \\
&\leq C e^n \norm{\Lambda_{q_1} - \Lambda_{q_2}} \sup_{z_0} \norm{u^{(1)}_{z_0}}_{H^1(\O)} \sup_{z_0} \norm{u^{(2)}_{z_0}}_{H^1(\O)} \\
&\leq c''_{p,M} \norm{\Lambda_{q_1} - \Lambda_{q_2}} e^{11n}.
\end{split}
\end{equation}
For the second term we need to show that $(f^{(1)} f^{(2)} - 1) \in X_\theta$. But notice that $f^{(1)} f^{(2)} - 1 = (f^{(1)} - 1)(f^{(2)} - 1) + f^{(1)} - 1 + f^{(2)} - 1$ and by interpolating the operator $h \mapsto (f^{(2)} - 1) h$, $f^{(2)} - 1 \in C^0(\wbar{\O}, W^{1,p}(\O))$ we get
\begin{equation}
\begin{split}
\norm{f^{(1)} f^{(2)} - 1}_{X_\theta} &\leq \norm{f^{(1)} - 1}_{X_\theta} \big( \sup_{z_0} \norm{f^{(2)}-1}_{W^{1,p}} + 1\big) + \norm{f^{(2)} - 1}_{X_\theta} \\
&\leq \big(c_{r,p,M}(c_{p,M}+\pi^{1/p}+1) + c_{r,p,M}\big) n^{r-\frac{1}{p*}} \leq c'''_{r,p,M} n^{r-\frac{1}{p*}}.
\end{split}
\end{equation}
Next, use theorem \ref{errorTermIntegral}. Note that $0 < p \leq \frac{4}{2-\theta}$ implies $r = \frac{\theta}{4} \leq \frac{\theta}{2} + \frac{1}{p*} - 1$, so $1 - \theta + r - \frac{1}{p*} \leq -\frac{\theta}{2}$. Moreover $n \geq 1$, so
\begin{multline}
\label{simplifyExp}
\norm{\frac{2n}{\pi} \int_\O e^{inR} Q(f^{(1)}_{z_0} f^{(2)}_{z_0} - 1) dm(z) }_{L^2(\O,z_0)} \leq C_p n^{1-\theta}\norm{Q}_{A_\theta} \norm{f^{(1)} f^{(2)} - 1}_{X_\theta} \\
\leq c''''_{r,p,M} n^{1-\theta+r-\frac{1}{p*}} \leq c''''_{r,p,M} n^{-\theta/2}.
\end{multline}

Now we can combine the terms. Remember that $n = \frac{1}{22} \ln \norm{\Lambda_{q_1}-\Lambda_{q_2}}^{-1}$, $\norm{\Lambda_{q_1} - \Lambda_{q_2}} < C_{M,\varepsilon,p} < e^{-1}$ and $\frac{1}{2} \geq \frac{\theta}{2} e^{-1}$. Thus by lemma \ref{polynomialLogarithm} we have $\norm{\Lambda_{q_1} - \Lambda_{q_2}}^{1/2} \leq \big( \ln \norm{\Lambda_{q_1} - \Lambda_{q_2}}^{-1}\big)^{-\theta/2}$. Finally
\begin{equation}
\begin{split}
\norm{q_1 - q_2}_{L^2(\O)} &\leq C''_{r,\theta,p,M}\big( n^{-\theta/2} + \norm{\Lambda_{q_1} - \Lambda_{q_2}} e^{11n} + n^{-\theta/2}\big) \\
&\leq C'_{r,\theta,p,M}\big( \norm{\Lambda_{q_1} - \Lambda_{q_2}}^{1/2} + \big(\ln \norm{\Lambda_{q_1} - \Lambda_{q_2}}^{-1}\big)^{-\theta/2}\big) \\
&\leq C'_{M,\varepsilon,a} \big( \ln\norm{\Lambda_{q_1} - \Lambda_{q_2}}^{-1} \big)^{-\theta/2},
\end{split}
\end{equation}
because $r$, $\theta$ and $p$ are functions of $\varepsilon$ and $a$.
\end{proof}

\end{document}